\documentclass[11pt, reqno]{amsart}    

\usepackage{amsmath,amssymb,latexsym,mathtools}
\usepackage{graphicx,xcolor,mathrsfs}
\usepackage{cite}
\usepackage[utf8]{inputenc}
\usepackage[english]{babel}
\usepackage[font={scriptsize}]{caption}

\DeclareMathOperator{\supp}{supp}

\newtheorem{theorem}{Theorem}[section]
\newtheorem{lemma}[theorem]{Lemma}
\newtheorem{proposition}[theorem]{Proposition}
\newtheorem{remark}[theorem]{Remark}
\newtheorem{definition}[theorem]{Definition}
\newtheorem{corollary}[theorem]{Corollary}
\newtheorem*{main-theorem}{Main Theorem}
\newtheorem*{remark*}{Remark}

\numberwithin{equation}{section}

\usepackage{breqn}
\usepackage{enumitem}
\usepackage{mathrsfs}

\usepackage[colorlinks=true]{hyperref}
\hypersetup{urlcolor=blue, citecolor=red, linkcolor=blue}

\usepackage{scalerel,stackengine}
\stackMath
\newcommand\widecheck[1]{%
\savestack{\tmpbox}{\stretchto{%
  \scaleto{%
    \scalerel*[\widthof{\ensuremath{#1}}]{\kern-.6pt\bigwedge\kern-.6pt}%
    {\rule[-\textheight/2]{1ex}{\textheight}}%WIDTH-LIMITED BIG WEDGE
  }{\textheight}% 
}{0.5ex}}%
\stackon[1pt]{#1}{\scalebox{-1}{\tmpbox}}%
}

\newlist{abbrv}{itemize}{1}
\setlist[abbrv,1]{label=,labelwidth=1in,align=parleft,itemsep=0.1\baselineskip,leftmargin=!}

\let\olddefinition\definition
\renewcommand{\definition}{\olddefinition\normalfont}
\let\oldremark\remark
\renewcommand{\remark}{\oldremark\normalfont}

\newcommand{\R}{\mathbb{R}}
\newcommand{\N}{\mathbb{N}}

\newcommand{\bo}{\bigg(}
\newcommand{\bc}{\bigg)}
\newcommand{\Bok}{\bigg[}
\newcommand{\bok}{\Big[}
\newcommand{\Bck}{\bigg]}
\newcommand{\bck}{\Big]}
\newcommand{\qqquad}{\quad\qquad}

\renewcommand{\k}{\kappa}
\newcommand{\x}{\xi}
\newcommand{\vr}{\varphi_r}

\newcommand{\vn}{\varphi_k}
\newcommand{\pn}{\psi_k}

\newcommand{\un}{u_k^1}
\newcommand{\um}{u_k^2}
\newcommand{\vun}{v_k^1}
\newcommand{\vum}{v_k^2}

\newcommand{\m}{\mu}

\newcommand{\norm}[2]{\|#1\|_{{#2}}}

\newcommand{\e}{\mathcal{E}}
\renewcommand{\l}{\mathcal{L}}
\newcommand{\n}{\mathcal{N}}
\newcommand{\q}{\mathcal{Q}}
\renewcommand{\O}{\mathcal{O}}

\renewcommand{\b}{\beta}

\newcommand{\horm}[1]{\|#1\|_{H^\frac{s}{2}}}
\newcommand{\cop}[1]{\langle#1\rangle}

\DeclareMathOperator*{\argmax}{arg\,max}

\title[Solitary waves for dispersive equations]{Solitary waves for weakly dispersive equations with inhomogeneous nonlinearities}

\author[M{\ae}hlen]{Ola I.H. Maehlen}
\address{Department of Mathematical Sciences, Norwegian University of Science and Technology, 7491 Trondheim, Norway}
\email{ola.mahlen@ntnu.no}

\thanks{The author acknowledges the support from grant no. 250070 from the Research Council of Norway.}


\begin{document}

\maketitle

\begin{abstract}
We show existence of solitary-wave solutions to the equation
\begin{equation*}
u_t+ (Lu - n(u))_x = 0\,,
\end{equation*}
for weak assumptions on the dispersion $L$ and the nonlinearity $n$.
The symbol $m$ of the Fourier multiplier $L$ is allowed to be of low positive order ($s > 0$), while $n$ need only be locally Lipschitz and asymptotically homogeneous at zero. We shall discover such solutions in Sobolev spaces contained in $H^{1+s}$.
\end{abstract}

\section{Introduction}
A great deal of model equations for the evolution of water waves in one spacial dimension can be compactly written as
\begin{equation}\label{originalstrong}
u_t + (Lu - n(u))_x = 0\,,
\end{equation}
where the dispersion $L$ is a Fourier multiplier in space with real-valued symmetric symbol $m$, that is, 
\[\widehat{Lu}(\x) = m(\x)\hat{u}(\x), \]
and $n$ is a local nonlinear term. Solutions of \eqref{originalstrong} tend to enjoy a variety of qualitative properties of water, see \cite{MR3060183}, but our focus will be on the existence of \textit{solitary waves}. Traveling at constant velocity \(\nu\), these solutions take the form $(x,t)\mapsto u(x-\nu t)$, where $u(y)\to0$ as $|y|\to\infty$. For such solutions \eqref{originalstrong} means
\begin{equation}\label{original}
-\nu u + Lu - n(u) = 0\,,
\end{equation}
in light of the assumption that $u$ vanish at infinity. 

A common approach to prove solitary waves in equations of the form \eqref{original} is Lion's concentration-compactness method introduced in \cite{MR778970}. Weinstein used this in 1987 to prove existence and orbital stability in the case of a monomial nonlinearity and a symbol of order \(s \geq 1\) \cite{MR886343}. The limit \(s = 1\) is not only superficial: In \cite{MR1455330} the authors study an equation corresponding to \(s=1\), and that method was later put in a more general framework in \cite{MR1647189}, again for \(s \geq 1\). Zeng \cite{MR1954506} later used a different energy functional (and different conserved quantity) to relax some of the conditions, but still for \(s \geq 1\). 

These works led a number of different authors to consider the case when \(s < 1\): in \cite{MR3360393} and \cite{MR3485840} the authors treat equations with positive-order Fourier operators (\(s > 0\)) — the case of homogeneous and inhomogeneous symbols respectively – and in both cases with homogeneous nonlinearities; whereas in \cite{MR2979975} smoothing operators (\(s < 0\)) with mildly inhomogeneous nonlinearities are allowed.  The method for positive-order operator is indeed based upon Weinstein's paper \cite{MR886343}, whereas the method for negative-order operators is different, and more closely related to works on the Euler equations and other systems with dispersion of very weak type \cite{MR2847283}.  A main difference between the works \cite{MR3360393, MR3485840} and \cite{MR2979975} is the requirement that the waves in the latter should be small. This is related to scalings/homogeneity of the nonlinearity, and an essential part of the method of proof in \cite{MR2979975}. A later work, related to the investigations for positive \(s\), is \cite{MR3060817}, in which the authors look at \eqref{originalstrong} when the nonlinearity is polynomial, cubic or higher, and the symbol \(m\) grows at least as \(|\xi|^{\frac{1}{2}}\) at infinity. This growth may be slightly lowered: in the case of a quadratic pure-power nonlinearity and a homogeneous symbol \(m\) (the fractional KdV equation), the optimal assumption in terms of growth is \(m(\xi) = |\xi|^{p}\), \(p>\frac{1}{3}\) \cite{MR3070568}; below this value one does not have solitary waves for the (homogeneous) fKdV equation \cite{MR3188389}. This coincides with our assumption on \(s\) below; for the assumption on \(s'\), see our remarks in Section~\ref{subsec:symbol m}.

Our goal has been twofold. First, to combine ideas from \cite{MR3485840} and \cite{MR2979975} to allow for more inhomogeneous nonlinearities in the theory for lower-order (\(s > 0\)) symbols; and, second, to improve upon the required assumptions on both the linear and nonlinear terms by a slightly different method of proof. The last point is made visible mostly in that the theory for low-order \(s\) is carried out in corresponding low-order Sobolev spaces (below the \(L^\infty\) embedding), for which we use a cut-off of the nonlinearity $n$ which is different from the `small ball' used in \cite{MR2979975}. (Our solutions will eventually be somewhat more regular, but the near-minimizers we work with might not exhibit the same regularity). In effect, we are able to reduce the assumptions on \eqref{original} to the following.

\subsection{The assumptions and the main theorem }\label{assumptions}
Throughout the paper, we will assume the following:
\begin{enumerate}
    \item [(A)]  The nonlinearity $n\colon \R\to\R$ is locally Lipschitz, and decomposes into $n = n_p + n_r$, where $n_p$ is homogeneous of one of the two forms:\\[-8pt]
    \begin{enumerate}[label=(\alph*)]
         \item[(A1)]  $x\mapsto c|x|^{1+p}$ and $c\neq 0$,\\[-10pt]
          \item[(A2)]  $x\mapsto cx|x|^{p}$ and $c>0$,\\[-8pt]
    \end{enumerate}
    for a real number $p > 0$, while the remainder term satisfies $n_r(x)=\mathcal{O}(|x|^{1+r})$, as $x\to0$, for some $r> p$.\\
    \item[(B)]  The symbol $m\colon\R\to\R$ is even and satisfies the growth bounds
\begin{equation*}
\begin{cases}
 m(\xi)-m(0)\simeq|\xi|^{s'}, & \text{for } |\x|<1,\\
 m(\xi)-m(0)\simeq|\xi|^{s}, & \text{for } |\x|>1,
 \end{cases}
\end{equation*}
with $s'> p/2$ and $s> p/(2+p)$. We also require \(\xi\mapsto m(\xi)/\cop{\x}^{s}\) to be uniformly continuous on \(\R\).
\end{enumerate}

We will discuss these assumptions in detail below. Given them, we will prove the following existence result.

\begin{theorem}\label{main}
There exist $\mu_*>0$ so that
for every $\mu\in(0,\mu_*)$, there is a solution $u\in H^{1+s}$ of \eqref{original}, with wave speed $\nu\in\R$, satisfying\\[-10pt]
\begin{enumerate}[label=(\roman*)]
    \item $\norm{u}{H^{1+s}}^2\lesssim\norm{u}{2}^2 = 2\mu$,\\[-4pt]
     \item 
        $m(0)-\nu\simeq\m^{\b}$, with $\b= \frac{s'p}{2s'-p}$,
\end{enumerate}
where the implicit constants in $(i)$ and $(ii)$ are independent of $\mu\in(0,\mu_*)$.
\end{theorem}

An interesting special case of Theorem~\ref{main} is the case of the capillary-gravity Whitham equation with strong surface tension, for which \(p=1\) and the symbol is 
\[
m(\xi) = \textstyle \left( (1+ T \xi^2) \frac{\tanh(\xi)}{\xi} \right)^{\frac{1}{2}}, \qquad T \geq \frac{1}{3},
\]
which corresponds to \(s = \frac{1}{2}\) and $s'=2$. Modelled on the water wave problem with surface tension, the capillary-gravity Whitham equation is known to admit generalized solitary waves in the case \(T < \frac{1}{3}\) (weak surface tension) \cite{doi:10.1111/sapm.12288}, and decaying solitary waves for \(T > 0\) (both weak and strong surface tension) \cite{MR3485840},  as well as periodic steady waves, including rippled solutions in the case of weak surface tension \cite{EJMR18}. In the case \(T < \frac{1}{3}\) the solitary waves have wave speeds \(\nu\) smaller than \(m(0)\) (called subcritical), whereas the generalized waves exhibit supercritical wave speeds \(\nu > m(0)\); for strong surface tension we are only aware of sub-critical solutions.  As we also prove the existence of sub-critical solutions, in the case of strong surface tension \(T \geq \frac{1}{3}\), there currently seems to lack super-critical truly solitary waves in the capillary-gravity Whitham equation. The same waves have also not been found for the capillary-gravity Euler equations (although we have not found a source actually stating this), but a proof of general non-existence is lacking. What has been shown is that there are no small-amplitude, exponentially decaying, even, supercritical solitary-wave solutions of the Euler equations in the slightly weak case when \(T\) is close to, but less than, \(\frac{1}{3}\) \cite{MR1702734}.

On a related note, it might be worth noticing that Theorem~\ref{main} is also an existence result for solitary waves tending to a general value $c$, not necessarily zero, at infinity. For if $\tilde{n}(x)=  n(c+x)-n'(c)x-n(c)$ satisfies the assumptions, then there is a solitary-wave solution $u$, with velocity $\nu$, of the equation $u_t + (Lu-\tilde{n}(u))_x=0$, and thus, $u+c$ is a traveling wave solution of \eqref{original} with velocity $\nu-n'(c)$.

\subsection{The method}\label{the method} In this subsection, the framework used to prove Theorem \ref{main} will be introduced. In particular, we develop a constrained minimization problem whose solutions satisfy \eqref{original}, and in fact, it is exactly solutions of this minimization problem that we shall prove the existence of. For this purpose, we will be working with two `extra' assumptions on \eqref{original}, namely
\begin{enumerate}
    \item [(C\textsubscript{1})]  $n$ is globally Lipschitz continuous,
    \item[(C\textsubscript{2})] $m(0)=0$.
\end{enumerate}
While these auxiliary assumptions (especially the first) excludes many instances of \eqref{originalstrong} where we would like to prove the existence of solitary wave solutions, it turns out that proving our main theorem for this smaller class implies the result in the more general setting, as we now demonstrate.

\begin{lemma}\label{equivalenceOfAssumptions}
If Theorem \ref{main} holds true under the assumptions (A), (B), (C\textsubscript{1}) and (C\textsubscript{2}), then it also holds true when only (A) and (B) are satisfied.
\end{lemma}
\begin{proof}
Assume $n$ and $m$ satisfy (A) and (B). Define 
\begin{align*}
    \tilde n (x) &= \begin{cases}
     n(x),\quad |x|\leq1,\\
     n(\pm 1), \quad \pm x > 1,
    \end{cases}
    & \tilde{m} (\xi)&=m(\xi)-m(0),
\end{align*}
and notice that $\tilde n$ and $\tilde m$ satisfy (A), (B), (C\textsubscript{1}) and (C\textsubscript{2}). By assumption, Theorem \ref{main} now holds for the modified equation
\begin{align*}
    -\tilde \nu u + \tilde L u - \tilde n (u)=0,
\end{align*}
where $\tilde L$ is the Fourier multiplier whose symbol is $\tilde m$. Thus there is a $\tilde{\m}_*>0$ so that for each $\m\in(0,\tilde\m_*)$ we have a solution $u$ with velocity $\tilde \nu$ satisfying 
\begin{align*}
    \norm{u}{H^{1+s}}^2&\lesssim\mu,\\
     -\tilde{\nu} &\simeq \m ^\b,
\end{align*}
where we omitted $\tilde m(0)=0$ from the second expression. As $H^{1+s}\hookrightarrow L^\infty$, we can pick $\m_*\in (0,\tilde\m_*)$ so that $\norm{u}{\infty}\leq 1$ for all $\mu\in(0,\m_*)$. For such solutions $u$, we have $\tilde n(u) = n(u)$, and setting $\nu= \tilde{\nu}-m(0)$ we see that
\begin{align*}
    0&= -\tilde \nu u + \tilde L u - \tilde n (u),\\
    &= -\nu u + (\tilde L +  m(0) )u -  n (u),\\
    &= -\nu u +  Lu -  n (u).
\end{align*}
Thus, for  $\mu<\mu_*$ the solutions provided by Theorem \ref{main} for the modified equation are solutions of the original equation, but with a shifted velocity $\nu$ satisfying 
\begin{align*}
    m(0)-\nu\simeq\m^\b.
\end{align*}
\end{proof}

We now construct the minimization problem mentioned above, whose well-posedness is assured when the assumption (C\textsubscript{1}) is added to (A) and (B).
We will work in the Sobolev space $H^{\frac{s}{2}}$ of measurable functions $f \colon \R\to\R$ with finite Sobolev norm 
\begin{align*}
    \horm{f} = \norm{\cop{\cdot}^{\frac{s}{2}}\hat{f}}{2},
\end{align*}
where we use the Japanese bracket  \(\cop{\x}= \left( 1+\xi^2 \right)^{1/2}\). Our main tools shall be the functionals  $\q,\l,\n \colon H^{\frac{s}{2}}\to\R$, defined by 
\begin{align*}
    \q(u)&= \frac{1}{2}\int_\R u^2\,dx, \\
    \l(u) &=  \frac{1}{2}\int_\R m(\x)|\hat{u}|^2\,d\x,\\
    \n(u)= \n_p(u)+\n_r(u)&= \int_\R N_p(u)\,dx + \int_\R N_r(u)\,dx,
\end{align*} where $N_p(x)= \int_0^xn_p\,dt$, and $N_r(x)= \int_0^xn_r\,dt$. We will prove the above functionals to be Fréchet differentiable with \(H^{\frac{s}{2}}\)-derivatives 
\[
\q'(u)=u, \qquad \l'(u)=Lu, \qquad\text{and}\qquad \n'(u)=n(u).
\]
Consider now the constraint minimization problem
\begin{align}\label{problem}
    I_\mu =  \inf_{u\in U_{\m}}\e(u)\,,
\end{align}
 where $\e= \l-\n$ and
 \begin{align}\label{ConstrainSet}
 U_\mu = \{u\in H^{\frac{s}{2}} \colon \q(u)=\mu\}, 
 \end{align}
and where we restrict $\mu\in(0,\mu_*)$, for some fixed upper bound $\mu_*$ that we shall require to be sufficiently small. Our strategy shall be to find minimizers of \eqref{problem}; a minimizer $u$ must for some Lagrange multiplier $\nu\in\R$ satisfy 
 \[
 0=-\nu\q'(u)+\e'(u) = -\nu u+Lu-n(u), 
 \]
 thus solving \eqref{original}.  Note that, although our solutions are `discovered' in $H^{\frac{s}{2}}$, we additionally prove they lie in the more regular space $H^{1+s}$ (or, in an even more regular space, see Prop. \ref{furtherRegularity}). Had we been working on a compact domain, then any ``uniformly regular" minimizing sequence of \eqref{problem}, would admit a converging subsequence, implying the existence of a minimizer. As $\R$ is not compact, we instead use Lion's concentration--compactness theorem (see Section~\ref{preliminariesfunctionalsection}). Informally, any bounded sequence $(\rho_k)\subset L^{1}$ admits a subsequence (again indexed with $k$) that will, as \(k \to \infty\), either
 \begin{itemize}
    \itemsep -0.7em 
     \item[--] \emph{vanish} (the mass spreads out),\\
     \item[--] \emph{dichotomize}  (the mass splits in two parts that separate), or\\
     \item[--] \emph{concentrate} (the mass remains uniformly concentrated in space).
 \end{itemize} 
We will show that for a `concentrated' minimizing sequence, we can pick a converging subsequence. Thus, the existence of a minimizer of \eqref{problem} follows if we can for minimizing sequences rule out the possibility of vanishing and dichotomy. To achieve this, we use a ``long-wave ansatz" to find a low enough upper bound for $I_\mu$ that will allows us to compare the size of $\mu$, $\l$ and $\n$ on `near minimizers'. This size comparison will directly exclude vanishing and also imply that $\mu\mapsto I_\mu$ is subadditive for small $\mu>0$, which excludes dichotomy. The paper concludes with some regularity estimates for our solutions (see Prop. \ref{furtherRegularity}).

We end this section with some discussion regarding the main assumptions (A) and (B).

\subsection{A technical look at the assumptions (A) and (B)}
In this subsection, we discuss our main assumptions on the the pair $n$ and $m$; we mention what role the different parts play and whether some could be weakened. This discussion is easier to follow after a read through.

\subsubsection{The nonlinearity \(n\)} 
The continuity of $n$ is needed for $\n$ to be Fréchet differentiable. The stronger local Lipschitz continuity is used to obtain the estimate $\norm{u}{H^{1+s}}^2\lesssim\mu$ for our solutions in Prop.~\ref{regularityofsolutions}; this important estimate gives us Lemma \ref{equivalenceOfAssumptions} which is what we use to guarantee the well-posedness of \eqref{problem} in the case $s\leq1$. Still, there are two alternative ways of proving solitary waves when we assume $n$ to be merely continuous:\\[-10pt]

\begin{itemize}
\item[(i)] If $s>1$, we have $H^{\frac{s}{2}} \hookrightarrow BC$, and so one could use Prop.~\ref{complexbound} (specifically equation \eqref{hormbound}) in place of Prop.~\ref{regularityofsolutions} to attain Lemma \ref{equivalenceOfAssumptions}.\\[-8pt]

\item[(ii)] Alternatively, if $| n_r(x) |\lesssim |x|^{1+p}$ for $|x|>1$, all steps in this paper (apart from Prop. \ref{regularityofsolutions}) go through, granted we include the restriction $\horm{u}<R$ to our minimization problem for some arbitrary constant $R>0$, which only plays a role in proving Prop.~\ref{complexbound}.\\[-10pt] 
\end{itemize}
We choose to assume local Lipschitz continuity of \(n\) to avoid these other conditions, and to provide a somewhat different technique in comparison to earlier proofs.

Finally, the reason for excluding the case $n_p(x)=cx|x|^p$, $c<0$, is the same as in \cite{MR3485840} and \cite{MR2979975}. Our method breaks down at the first step in that regime, as we cannot hope to obtain the low upper bound for $I_\mu$ in Prop.~\ref{bounds}, because $-\n_p(u)>0$ for all $u\neq 0$. 

\subsubsection{The symbol $m$}\label{subsec:symbol m}
The upper bound of the growth at zero and the corresponding inequality $s'>p/2$ are needed to find a satisfactorily low upper bound for $I_\mu$ by a long-wave ansatz (see Prop.~\ref{bounds}), while the lower bound is necessary for Prop.~\ref{complexbound}, which is crucial for the remainder term $n_r$ to be negligible for sufficiently small $\mu$. 

As for the growth bounds when $|\xi|>1$, the lower bound is chosen to control the $H^{\frac{s}{2}}$-norm by $\q$ and $\l$, which together with $s>p/(2+p)$ gives control of the $L^{2+p}$-norm by Sobolev embedding. This is used in the proof of Prop. \ref{complexbound} and in \eqref{congestionfromgaglierdo} to exclude vanishing. 

The upper growth bound is instead needed when excluding dichotomy: Indeed, if $m(\cdot)-m(0)$ was bounded by $\cop{\cdot}^{\tilde{s}}$, $\tilde{s}>s$, we would need to work in $H^{\tilde{s}/2}$ (for $\e(u)$ to be well defined). Then equation \eqref{hormbound}, which bounds the $H^{\frac{s}{2}}$-norm, would still be the best regularity estimate on a minimizing sequence, but Lemma \ref{zous} (now, for operators $B_r\colon H^{\tilde{s}/2}\to H^{-\tilde{s}/2}$), would require a bound on the stronger $H^{\tilde{s}/2}$-norm to be of any use when proving Prop.~\ref{spezial}. 

Finally, the uniform continuity of $\xi\mapsto m(\xi)/\cop{\xi}^s$ is necessary for excluding dichotomy. It assures that $L$ is not `too' non-local, as described in Lemma~\ref{zous}. Note that a sufficient estimate for our regularity constraint is $|m'(\xi)|\lesssim\cop{\xi}^s$, as it implies that $\xi\mapsto m(\xi)/\cop{\xi}^s$ is globally Lipschitz.

\section{Preliminaries}\label{preliminariesfunctionalsection}
In this section, we presents bounds and regularity estimates for the functionals $\q,\l,\n,\e$ introduced in subsection \ref{the method}. Throughout section 2-7, we assume (only) that $n$ and $m$ satisfies the assumptions (A), (B), (C\textsubscript{1}) and (C\textsubscript{2}), introduced in subsection \ref{assumptions} and \ref{the method}. In light of Lemma \ref{equivalenceOfAssumptions}, proving Theorem \ref{main} in this case, implies the validity of the theorem when either (C\textsubscript{1}) or (C\textsubscript{2}) fails. 
\begin{proposition}\label{basicbounds}
For $u\neq 0$, we have
\begin{align*}
    (i)&\,0<\l(u)\lesssim \horm{u}^2, &  (iii)&\,|\n_p(u)|\lesssim\norm{u}{2+p}^{2+p},\\
       (ii)&\,|\n(u)|\lesssim \q(u), &
    (iv)&\, |\n_r(u+v)|\lesssim\norm{u}{2+r}^{2+r} + \norm{v}{2+p}^{2+p}.
    \end{align*}
\end{proposition}
\begin{proof}
Combining the growth bounds on $m$ from (B) with (C\textsubscript{2}), we see that $0<m(\x)\lesssim\cop{\x}^s$ for $\x\neq 0$, and so bound $(i)$ follows. By (A) and (C\textsubscript{1}), we have $|n(x)|\lesssim |x|$, and so we obtain $(ii)$. From $|n_p(x)|\lesssim|x|^{1+p}$ we immediately get $(iii)$. For $(iv)$, we note that
\begin{align*}
    |N_r(x)|\lesssim |x|^{2+r}, \quad |x|\leq 1,\quad \text{and}\quad
    |N_r(x)|\lesssim |x|^{2+p}, \quad |x|\geq 1,
\end{align*}
where the the first bound follows from $n_r(x)=O(|x|^{1+r})$, while the latter follows from $|n_r(x)|=|n(x)-n_p(x)|\lesssim |x|+|x|^{1+p}$. With this, and the fact that $r>p$, we obtain
\begin{align*}
    |N_r(x)|\lesssim \min\{|x|^{2+r},|x|^{2+p}\},
\end{align*}
or equivalently
\begin{align*}
    \frac{|N_r(x+y)|}{|x|^{2+r}+|y|^{2+p}}\lesssim \min\Bigg\{\frac{|x+y|^{2+r}}{|x|^{2+r}+|y|^{2+p}},\frac{|x+y|^{2+p}}{|x|^{2+r}+|y|^{2+p}}\Bigg\}\eqqcolon\min\Big\{a(x,y),b(x,y)\Big\}.
\end{align*}
Note that $a(x,y)$ and $b(x,y)$ are bounded for $|y|\leq 1$ and $|y|\geq1$ respectively, and so $|N_r(x+y)|\lesssim|x|^{2+r}+|y|^{2+p}$.
\end{proof}

From here on, we will refrain from explicitly referring to the assumptions as done in the previous proof, so to attain a more straight forward presentation.

\begin{proposition}\label{frechetderivative}
The Fréchet derivative of $\q,\l,\n$ and $\e$ at $u\in H^{\frac{s}{2}}$ are the elements in the (dual) space $H^{\frac{-s}{2}}$ given by
\begin{enumerate}[label=(\roman*)]
\item $\q'(u) = u$,
\vspace{5px}
\item $\l'(u) = Lu$,
\vspace{5px}
\item $\n'(u) = n(u)$,
\vspace{5px}
\item $\e'(u) = Lu-n(u)$.
\end{enumerate}
\end{proposition}

\begin{proof}
The Fréchet derivative of $\q$ and $\e$ follows from an elementary calculation and linearity of the Fréchet derivative respectively. Turning to $\l$, we note that $L$ is self-adjoint, $\cop{Lu,v}=\cop{u,Lv}$, due to the symmetry of $m$. Consequently $\l(u+v)=\l(u) + \cop{Lu,v} + \l(v)$.
We then obtain
\begin{align*}
    \frac{|\l(u+v)-\l(u) - \cop{Lu,v}|}{\horm{v}}&=\frac{\l(v)}{\horm{v}}
    \lesssim \horm{v}
    \to0\,,
\end{align*}
as $v\to 0$, in $H^{\frac{s}{2}}$, where we used $(i)$ from Prop. \ref{basicbounds}. For $\n$, we exploit the global Lipschitz-continuity of $n$ and calculate
\begin{align*}
    \frac{|\n(u+v)-\n(u) - \cop{n(u),v}|}{\horm{v}}&\leq \frac{1}{\horm{v}}\int_\R |v|\int_0^1|n(u+tv)-n(u)|\,dt\,dx\\
     &\lesssim \frac{\norm{v}{2}^2}{\horm{v}}\to0,
\end{align*}
as $v\to 0$, in $H^{\frac{s}{2}}$.
\end{proof}
One important implication of the previous proposition is the following description of the continuity of $\e$ on $H^{\frac{s}{2}}$, that we shall utilize when excluding dichotomy. 

\begin{corollary}\label{continuityofe}
For $u,v\in H^{\frac{s}{2}}$ we have
\begin{align*}
|\e(u)-\e(v)|\lesssim(\horm{u}+\horm{v})\horm{u-v}.
\end{align*}
\end{corollary}
\begin{proof}
Using $|n(u)|\lesssim|u|$ and $m(\xi)\lesssim\cop{\xi}^s$, we have for arbitrary $u,v\in H^{\frac{s}{2}}$
\begin{align*}
    |\cop{\e'(u),v}|&\leq |\cop{Lu,v}|+ |\cop{n(u),v}|\\
    &\lesssim\horm{u}\horm{v}+ \norm{u}{2}\norm{v}{2}\lesssim\horm{u}\horm{v}.
\end{align*}
We then conclude
\begin{align*}
    |\e(u)-\e(v)|&\leq\max_{0\leq t\leq1}|\cop{\e'(v+(u-v)t),u-v}|\\
    &\lesssim(\horm{u}+\horm{v})\horm{u-v}.
\end{align*}
\end{proof}

The uniform continuity of $\xi\mapsto m(\xi)/\cop{\xi}^s$ is a simple assumption to state, but not directly convenient to work with. Instead we shall use an implied regularity constraint on $m$, described by the next lemma.
\begin{lemma}
There is a function $\omega\colon \R\to[0,\infty)$, bounded above by a polynomial, with $\lim_{t\to0}\omega(t)=0$, such that 
\begin{align}\label{moreconvenientbound}
    |m(\xi)-m(\eta)|\leq \omega(\xi-\eta) \cop{\xi}^{\frac{s}{2}}\cop{\eta}^{\frac{s}{2}}.
\end{align}
\end{lemma}
\begin{proof}
Firstly, the bound $|\cop{\x}^s-\cop{\eta}^s|\lesssim (\cop{\xi}^{s}+\cop{\eta}^{s})|\xi-\eta|$, is easily obtained by the mean value theorem together with crude upper bounds. By assumption, there is a modulus of continuity $\tilde{\omega}$ so that
\begin{align}\label{omegaintroduction}
    \Big|\frac{m(\xi)}{\cop{\xi}^s}-\frac{m(\eta)}{\cop{\eta}^s}\Big|\leq\tilde{\omega}(\xi-\eta),
\end{align} and $\lim_{\lambda\to0}\tilde{\omega}(\lambda)=0$. As $m(\cdot)/\cop{\cdot}^s$ is a bounded function, we can assume $\tilde{\omega}$ to also be bounded. We arrive at
\begin{align*}
    |m(\x)-m(\eta)|&\leq\Big|\frac{m(\xi)}{\cop{\xi}^s}-\frac{m(\eta)}{\cop{\eta}^s}\Big|\cop{\x}^s + \frac{m(\eta)}{\cop{\eta}^s}|\cop{\xi}^s-\cop{\eta}^s|\\
    &\lesssim \tilde{\omega}(\x-\eta)\cop{\x}^s + |\xi-\eta|(\cop{\xi}^{s}+\cop{\eta}^{s})\\
    &\lesssim\big(\tilde{\omega}(\x-\eta) + |\xi-\eta|\big)\cop{\xi-t}^{\frac{s}{2}}\cop{\xi}^{\frac{s}{2}}\cop{\eta}^{\frac{s}{2}},\\
    &\eqqcolon\omega(\xi-\eta)\cop{\xi}^{\frac{s}{2}}\cop{\eta}^{\frac{s}{2}},
\end{align*}
where we used the estimate $\cop{x}\lesssim\cop{x-y}\cop{y}$, when going from second to third line.
\end{proof}
By a more careful argument, it is possible to show that the two regularity constraints \eqref{moreconvenientbound} and \eqref{omegaintroduction} are equivalent without any a priori knowledge of $m$, although we shall not prove this.

We conclude this section with the concentration-compactness theorem; the foundation of our proof of Theorem \ref{main}.
\begin{theorem}[ Lions \cite{MR778970}, concentration-compactness]\label{concentration}
Any sequence $(\rho_k)\subset L^1$ of non-negative functions with the property
\begin{equation*}
    \int_\R \rho_k dx = \mu >0,
\end{equation*}
admits a subsequence, denoted again by $(\rho_k)$, for which one of the following phenomena occurs.\\
Vanishing: For each $r>0$, $k\to\infty$ implies that
\begin{equation*}
    \sup_{x_0\in\R}\int_{-r}^{r}\rho_k(x-x_0)dx\to0.
\end{equation*}
Dichotomy: There exist $\lambda\in(0,\mu)$, and sequences $(x_k)\subset\R$ and $(r_k),(\tilde{r}_k)\subset\R^+$, so that when $k\to\infty$
    \begin{align*}
  \int_{-r_k}^{r_k}\rho_k(x-x_k) dx&\to\lambda, & r_k&\to\infty, \\
    \int_{-\tilde{r}_k}^{\tilde{r}_k}\rho_k(x-x_k) dx&\to\lambda, & \tilde{r}_k/r_k&\to\infty,
\end{align*}
Concentration: There is a sequence $(x_k)\subset\R$, so that for each $\varepsilon>0$ there exists $r<\infty$ satisfying for all $k\in \N$
\begin{equation*}
    \int_{-r}^{r}\rho_k(x-x_k)\,dx \geq \mu -\varepsilon.
\end{equation*}
\end{theorem}

\section{Upper and lower bounds for $I_{\mu}$}\label{bounds}
In this section, we prove that the infimum $I_\mu$ of the minimization problem \eqref{problem} satisfies $-\infty<I_\mu< - \k\m^{1+\b}$, for two positive constants $\k$ and $\beta$. The upper bound will give us Prop. \ref{complexbound}, which declares some fruitful bounds on near minimizers. The importance of also having a lower bound is the trivial consequence $I_\mu\neq -\infty$, allowing Prop. \ref{subadditive} to be meaningful. For clarity, we note that $\mu_*$, as of now, is an arbitrary fixed positive upper bound for $\mu$.
The proof of the following proposition is inspired by \cite{MR2979975}.
\begin{proposition}\label{sizeImu}
There exists $\kappa>0$, so that for $\m\in(0,\mu_*)$, we have ${-\infty<I_{\mu}< - \kappa\mu^{1+\b}}$, where the exponent $\b=s'p/(2s'-p)$.
\end{proposition}
\begin{proof}
Note that $(i)$ and $(ii)$ in Prop. \ref{basicbounds}, immediately gives us that $I_\m>-C\mu$ for some $C<\infty$. For the upper bound, we pick a function $\varphi$, satisfying supp$(\hat{\varphi})\subset(-1,1)$, $\q(\varphi)=1$ and $c\varphi(x)\geq0$. This last inequality implies that $\n_p(\varphi) = \frac{|c|}{2+p}\norm{\varphi}{2+p}^{2+p}\,$. An example of such a function would be an appropriately scaled version of $x\mapsto$ sinc$(x)^2$. We define the ansatz function $\varphi_{\m,t}(x) = \sqrt{\frac{\m}{t}}\varphi(x/t)$, for $t\geq1$. By a substitution of variables we obtain
\begin{align}\label{knorm}
    \norm{\varphi_{\m,t}}{k}^k =\m\Bok\frac{\mu}{t}\Bck^{\frac{k}{2}-1}\norm{\varphi}{k}^k\,.
\end{align}
When $k=2$, we get $\q(\varphi_{\m,t})=\mu$, and moreover
\begin{align*}
\n_p(\varphi_{\m,t})&= \frac{|c|}{2+p}\norm{\varphi_{\m,t}}{2+p}^{2+p}\, \eqqcolon C_1\mu\Bok\frac{\mu}{t}\Bck^{\frac{p}{2}},\\
\n_r(\varphi_{\m,t}) &\lesssim \norm{\varphi_{\m,t}}{2+r}^{2+r} = \O(\m)\Bok\frac{\mu}{t}\Bck^{\frac{r}{2}}.
\end{align*}
Exploiting the local growth of $m$, a simple computation gives the inequality $\l(\varphi_{\m,t}) \leq C_2\m/t^{s'}$, for some $C_2<\infty$. We evaluate the ansatz to obtain
\begin{equation*}
\begin{split}
I_\m\leq\e(\varphi_{\m,t}) &\leq -\Bigg[C_1\Bok\frac{\mu}{t}\Bck^{\frac{p}{2}}-\frac{C_2}{t^{s'}}\Bigg]\m + \O(\mu)\Bok\frac{\mu}{t}\Bck^{\frac{r}{2}}\,.
\end{split}
\end{equation*}
We set $t^{-s'} =B\m^{\b}$ with $\b=s'p/(2s'-p)$, where $B>0$ is small enough to guarantee $t\geq1$ for $\mu\in(0,\mu_*)$. The inequality above becomes
\begin{equation*}
\begin{split}
I_\m&\leq - \underbrace{\bigg[C_1B^{\frac{p}{2s'}}-C_2B\bigg]}_{\let\scriptstyle\textstyle
    \substack{2\k}}\mu^{1+\b} + B^{\frac{r}{2s'}}\O\bo\mu^{1+\b +\frac{r-p}{2}}\bc\,.
\end{split}
\end{equation*}
Without loss of generality, we can choose $B$ small enough so that $\kappa>0$ and $\kappa\m^{1+\b}$ is greater than the $\O$-term for all values of $\mu\in(0,\m_*)$; this is possible as $p<\min\{2s',r\}$ and $\mu_*<\infty$ is fixed. We get the desired result: 

\begin{equation}\label{beautybound}
I_\m< \ - \kappa\mu^{1+\b}\,.
\end{equation}
\end{proof}
\begin{remark}
From here on, we assume to have picked a constant $\kappa>0$ as described in the last proposition. It is important to note that if we replace $\mu_*$ by a lower upper bound $\mu_*'<\mu_*$, then \eqref{beautybound} would still hold for the same $\kappa$, as $(0,\mu_*')\subset(0,\mu_*)$. This allows us to later assume $\mu_*$ to be `sufficiently' small, without having to worry about the effect on $\kappa$. Similarly, the implicit constants in Prop. \ref{complexbound} will also remain fixed when lowering $\m_*$.
\end{remark}

\section{Near minimizers}
A consequence of the preceding proposition is that the feasible region $U_\mu=\{u\in H^{\frac{s}{2}}:Q(u)=\mu\}$ of the the minimization problem \eqref{problem} contains elements $u$ satisfying
\begin{align*}
 \e(u)< - \kappa\mu^{1+\b},\qqquad \text{with }\b=\frac{s'p}{2s'-p},
\end{align*}
where $\kappa$ is some fixed positive constant independent of $\mu\in(0,\mu_*)$.
We will refer such functions as \textit{near minimizers}. Only these functions are of interest to us; any minimizing sequence $(u_k)\subset U_\mu$ must consist solely of near minimizers, except for a finite number of exceptions. Proposition \ref{complexbound} will give important bounds of such functions, that will serve as the main building blocks for excluding vanishing and dichotomy.  We stress that throughout this paper, the implicit constants associated with our usage of $\lesssim,\gtrsim$ and $\simeq$ are independent of $\mu\in(0,\m_*)$.
\begin{proposition}\label{complexbound}
A near minimizer $u\in U_\mu$ satisfies
\begin{align}
\label{equalbounds}
    \l(u)\simeq\n(u)\simeq\norm{u}{2+p}^{2+p}&\simeq\mu^{1+\beta}, \\
    \label{rapiddecay}
    \n_r(u) &= o(\mu^{1+\beta}),\\
    \label{hormbound}
\horm{u}^2&\simeq\mu.
\end{align}
\end{proposition}
\begin{proof}
\textit{Obtaining the bounds \eqref{equalbounds}.}
As $\l > 0$, we immediately get from the definition of a near minimizer that
\begin{align}\label{obviousinequality}
  \max\{\l(u),\mu^{1+\beta}\}\lesssim\n(u)\lesssim\norm{u}{2+p}^{2+p}, 
\end{align}
 where the last inequality follows from Prop. \ref{basicbounds}. It remains to show $\norm{u}{2+p}^{2+p}\lesssim\min\{\l(u),\mu^{1+\beta}\}$. Let the indicator function on $[-1,1]$ be denoted $\chi$ and partition $u=u_1+u_2$ with $\widehat{u_1}= \chi\hat{u}$ and $\widehat{u_2}= (1-\chi)\hat{u}$.  By the Gagliardo--Nirenberg interpolation inequality,
\begin{align}\label{fundamentalu1}
    \norm{u_1}{2+p}^{2+p}\lesssim \norm{u_1}{\dot{H}^{\frac{s'}{2}}}^{\frac{p}{s'}}\norm{u_1}{2}^{2+p-\frac{p}{s'}}\lesssim\l(u)^{\frac{p}{2s'}}\mu^{1+\frac{p}{2} - \frac{p}{2s'}}.
\end{align}
For $u_2$, we use Sobolev embedding to obtain
\begin{align}\label{fundamentalu2}
    \norm{u_2}{2+p}^{2+p}\lesssim\norm{u_2}{H^{\frac{s}{2}}}^{2+p}\lesssim\l(u)^{1+\frac{p}{2}}.
\end{align}
As $\l(u)\lesssim\n(u)$, and $\n(u)\lesssim\m$ by $(ii)$ in Prop. \ref{basicbounds},  the expression \eqref{fundamentalu2} can be reduced further to
\begin{align}\label{u2finished}
\norm{u_2}{2+p}^{2+p}\lesssim\l(u)^{\frac{p}{2s'}}\mu^{1+\frac{p}{2} - \frac{p}{2s'}}.    
\end{align}
Exploiting the connection $1+\frac{p}{2} - \frac{p}{2s'} = (1-\frac{p}{2s'})(1+\b)$, we combine inequality \eqref{fundamentalu1} and \eqref{u2finished} to obtain
\begin{align}\label{ufundamental}
    \norm{u}{2+p}^{2+p}\lesssim\norm{u_1}{2+p}^{2+p}+\norm{u_2}{2+p}^{2+p}\lesssim\l(u)^{\frac{p}{2s'}}\bok\mu^{1+\b}\bck^{1-\frac{p}{2s'}}.
\end{align}
Combining \eqref{obviousinequality} with \eqref{ufundamental}, we conclude that $\norm{u}{2+p}^{2+p}\lesssim\min\{\l(u),\m^{1+\b}\}$.

\textit{Obtaining the bound \eqref{rapiddecay}.}
Now that \eqref{equalbounds} is established, we get $\norm{u_1}{2+p}^{2+p}\lesssim\m^{1+\beta}$ by \eqref{fundamentalu1}. Moreover, $\norm{u_1}{\infty}^2\leq\norm{\widehat{u_1}}{1}^2\leq4\mu$, and so
\begin{align*}
    \norm{u_1}{2+r}^{2+r}\leq \norm{u_1}{2+p}^{2+p}\norm{u_1}{\infty}^{r-p}\lesssim\mu^{1+\beta + (r-p)/2}.
\end{align*}
Looking back at \eqref{fundamentalu2}, we also obtain $\norm{u_2}{2+p}^{2+p}\lesssim\mu^{(1+\frac{p}{2})(1+\beta)}$.
Finally, by $(iv)$ in Prop. \ref{basicbounds},
\begin{align*}
    |\n_r(u)|\lesssim \norm{u_1}{2+r}^{2+r} + \norm{u_2}{2+p}^{2+p} = o(\mu^{1+\beta}).
\end{align*}

\textit{Obtaining the bound \eqref{hormbound}.}
This is also a consequence of \eqref{equalbounds} together with $\horm{\cdot}^2\simeq\q(\cdot)+\l(\cdot)$ and the fact that the upper bound $\mu_*$ is fixed.
\end{proof}

\section{A congestion result for near minimizers}\label{vanishing}
In this section, we show that a minimizing sequence $(u_k)$ of \eqref{problem} will never \textit{vanish} in accordance with the Concentration-Compactness Theorem \ref{concentration}. We start by demonstrating some `uniform' congestion of mass in $L^{2+p}$-norm of each element in $(u_k)$. To formalize, we pick a smooth function $\varphi$, satisfying  $\supp(\varphi)\subset[-1,1]$ and $\sum_{j\in\mathbb{Z}}\varphi(x-j)=1$. An example would be the convolution of the characteristic function on $[-\frac{1}{2},\frac{1}{2}]$ with a mollifier supported in $[-\frac{1}{4},\frac{1}{4}]$. For brevity, we set $\varphi_j(x)=\varphi(x-j)$.
\begin{proposition}\label{congestion}
For any near minimizer $u\in U_\mu$ we have
\begin{align*}
    \max_{j\in\mathbb{Z}}\norm{\varphi_j u}{2+p}\gtrsim\mu^{\frac{\beta}{p}}.
\end{align*}
\end{proposition}
\begin{proof}
Consider the operator $T\colon f\mapsto(\varphi_jf)_j$, mapping functions to sequences of functions. It is a fact that $\norm{T}{H^{\alpha}\to\ell^2(H^{\alpha})}<\infty$ for all $\alpha\geq0$; this is a trivial calculation when $\alpha\in\mathbb{N}_0$ if one replaces $\norm{\cdot}{H^\alpha}$ with the equivalent norm $f\mapsto\norm{f}{2}+\norm{f^{(\alpha)}}{2}$. For non-integer values of $\alpha>0$, the result follows immediately from the (so called) `complex interpolation method'; in particular, the two results \cite[Theorem~5.1.2. on p. 107]{interpolation} and \cite[Theorem~6.4.5.(7) on p. 152]{interpolation} combined with the boundness of $T$ for $\alpha\in\mathbb{N}_0$, implies the general bound. Setting $\alpha=s/2$, we conclude
\begin{align}\label{boundnessOfT}
    \sum_{j\in\mathbb{Z}}\horm{\varphi_j u}^2\lesssim\horm{u}^2.
\end{align}
By \eqref{hormbound} and \eqref{equalbounds} we also obtain 
\begin{align}\label{equivalenceOfExpressions}
    \m^{\b}\horm{u}^2\simeq\norm{u}{2+p}^{2+p}\simeq\sum_{j\in\mathbb{Z}}\norm{\varphi_j u}{2+p}^{2+p},
\end{align}
where the last equivalence uses $\sum_{j\in\mathbb{Z}}|\varphi_j(x)|^{2+p}\simeq 1$.
Combining \eqref{boundnessOfT} and \eqref{equivalenceOfExpressions}, we get
\begin{align*}
\mu^\b\sum_{j\in\mathbb{Z}}\horm{\varphi_j u}^2\leq C\sum_{j\in\mathbb{Z}}\norm{\varphi_j u}{2+p}^{2+p},    
\end{align*}
for some $C<\infty$ independent of our choice of near minimizer $u$. At least one $j_0\in\mathbb{Z}$ must then satisfy 
\begin{align}\label{summandinequality}
    \mu^\b\horm{\varphi_{j_0} u}^2\leq C\norm{\varphi_{j_0} u}{2+p}^{2+p}.
\end{align}
Combining \eqref{summandinequality} with the Sobolev embedding, $\norm{\varphi_{j_0} u}{2+p}^2\lesssim\horm{\varphi_{j_0} u}^2$, we are done.
\end{proof}
To exclude vanishing we would need congestion of mass in $L^2$-norm; this is achievable from the previous result through the Gagliardo--Nirenberg inequality inequality. Indeed, setting $j_0=\argmax_{j\in\mathbb{Z}}\norm{\varphi_j u}{2+p}$ we obtain
 \begin{align}\label{congestionfromgaglierdo}
    \norm{\varphi_{j_0}u}{2+p}^{2+p}\lesssim\norm{\varphi_{j_0}u}{\dot{H}^{\frac{s}{2}}}^{\frac{p}{s}}\norm{\varphi_{j_0}u}{2}^{2+p-\frac{p}{s}}.
\end{align}
By the boundness of $T$ in the previous proof, and \eqref{hormbound}, we have 
the estimate $\norm{\varphi_{j_0}u}{\dot{H}^{\frac{s}{2}}}^2\lesssim\mu$; together with the previous proposition, equation \eqref{congestionfromgaglierdo} now implies
 \begin{align*}
    \mu^{\frac{\beta}{p}(2+p)}\lesssim\mu^{\frac{p}{2s}}\norm{\varphi_{j_0}u}{2}^{2+p-\frac{p}{s}}.
\end{align*}
As $2+p-p/s>0$, we conclude that $\mu^{\delta}\lesssim \norm{\varphi_{j_0}u}{2}$, for some appropriate exponent $\delta>0$, and so we get the following corollary.
\begin{corollary}
No minimizing sequence of \eqref{problem} has a subsequence for which \textit{vanishing} occurs in accordance with Theorem \ref{concentration}.
\end{corollary}

\section{Strict subadditivity of the mapping \texorpdfstring{$\m\mapsto I_{\m}$}{TEXT} }
Excluding \textit{dichotomy} from a minimizing sequence is a more difficult task than that of vanishing, reflected by the laborious calculations in this subsection. The main idea however, is a simple one: Suppose dichotomy (as described in Theorem \ref{concentration}) occurs on a minimizing sequence $(u_k)\subset U_\mu$ of \eqref{problem}, then we shall see it can be `split' in two $(u_k^1)\subset U_\lambda$, $(u_k^2)\subset U_{\mu-\lambda}$ so that $\lim_{k\to\infty}\e(u_k^1)+\e(u_k^2)=I_{\mu}$. This will contradict that the mapping $\mu\mapsto I_\mu$ is strictly subadditive for small $\mu$, a fact we now prove.

\begin{proposition}\label{subadditive}
For $\mu_*>0$ sufficiently small, the mapping $\m\mapsto I_\m$ is strictly subadditive on $(0,\mu_*)$, that is,
\begin{align*}
    I_{\m_1+\m_2}<I_{\m_1}+I_{\m_2},
\end{align*}
for $\m_1,\mu_2>0$ satisfying $\mu_1+\m_2< \m_*$.
\end{proposition}

\begin{proof}
We begin by finding a $\m_*>0$ so that $\mu\mapsto I_\mu$ is strictly subhomogenous on $(0,\mu_*)$.
Pick a near minimizer $u\in U_\mu$ and $t\in[1,2]$. Notice that $\l(\sqrt{t}u) = t\l(u)$ and $\n_p(\sqrt{t}u) = t^{1+\frac{p}{2}}\n_p(u)$. As $\q(\sqrt{t}u)=t\mu$, we calculate
\begin{equation}\label{inequalitydichotomy}
\begin{split}
I_{t\m} &\leq \l(\sqrt{t}u) - \n(\sqrt{t}u)\\
&= t\l(u)-t^{1+\frac{p}{2}}\n(u) +  t^{1+\frac{p}{2}}\n_r(u)-\n_r(\sqrt{t}u)\\
 &=t\e(u) - \underbrace{[t^{1+\frac{p}{2}}-t]\n(u)}_{\let\scriptstyle\textstyle
    \substack{\\
    \varphi(t,u)}} + \underbrace{t^{1+\frac{p}{2}}\n_r(u) - \n_r(\sqrt{t}u)}_{\let\scriptstyle\textstyle
    \substack{\phi(t,u)}}
\end{split}
\end{equation}
By \eqref{equalbounds} we get $\varphi(t,u)\gtrsim(t-1)\mu^{1+\b}$, where we  exploited that $t^{1+\frac{p}{2}}-t\gtrsim t-1$, when $t\in[1,2]$.
As for $\phi$, we see that $\phi(1,u)=0$ and so we use the mean value theorem for some $t_*\in[1,t]$ (and Leibniz integral rule) to get
\begin{align*}
    \phi(t,u) &= (t-1)\frac{d\phi}{dt}(t_*,u)\\
    &=(t-1)\int_\R(1+\tfrac{p}{2})t_*^{\frac{p}{2}}N_r(u)-\frac{u}{2\sqrt{t_*}}n_r(\sqrt{t_*}u)\,dx.
\end{align*}
It should be clear that $u\mapsto\int_\R un_r(\sqrt{t}u)\,dx$ also satisfies an inequality of the form $(iv)$ in Prop. \ref{basicbounds}, uniformly in $t\in[1,2]$. This in turn means it satisfies an inequality of the form \eqref{rapiddecay} uniformly in $t\in[1,2]$. Thus the above calculation implies that $|\varphi(t,u)|=(t-1)o(\mu^{1+\b})$. 
These two bounds on $\varphi$ and $\phi$ implies we can pick $\mu_*>0$ small enough so that
\begin{align*}
    -\varphi(t,u)+\phi(t,u)\leq -\delta(t-1)\mu^{1+\b},
\end{align*}
is satisfied for some $\delta>0$, all $t\in[1,2]$ and all near minimizers $u\in U_\mu$ with $\mu\in(0,\mu_*)$. Assuming we have chosen such a $\mu_*>0$, then \eqref{inequalitydichotomy} becomes 
\begin{align*}
    I_{t\mu}\leq t\e(u)-\delta(t-1)\mu^{1+\b}.
\end{align*}
Picking a minimizing sequence $(u_k)\subset U_\mu$ and assuming $1<t\leq2$, this last inequality implies
\begin{align}\label{closebutnocigar}
    I_{t\mu}<tI_\mu,
\end{align}
on $(0,\mu_*)$. Finally, for a general $t>1$ and $\mu$ satisfying $t\mu\in(0,\mu_*)$, we can pick an integer $k>0$, so that $\sqrt[k]{t}\leq 2$, which combined with \eqref{closebutnocigar} implies
\begin{align*}
    I_{t\mu}<t^{\frac{1}{k}}I_{t^{1-\frac{1}{k}}\mu}<t^{\frac{2}{k}}I_{t^{1-\frac{2}{k}}\mu}<\dots<tI_\mu,
\end{align*}
that is, $\mu\mapsto I_\mu$ is strictly subhomogenous on $(0,\mu_*)$.
To show that strict subhomogeneity implies strict subadditivity, we assume without loss of generality that $0<\m_1\leq\m_2$ and $\m_1+\m_2<\mu_*$, and calculate
\begin{align*}
    I_{\m_1+\m_2}<\Big(\frac{\m_1}{\m_2}+1\Big)I_{\m_2}
    =\frac{\m_1}{\m_2}I_{\frac{\m_2}{\mu_1}\m_1} + I_{\m_2} 
    \leq I_{\m_1} + I_{\m_2}\,.
    \end{align*}
\end{proof}

Now that strict subadditivity of $\mu\mapsto I_\mu$ has been established, we shall create the contradiction as described at the beginning of this section. It will be essential that the non-local component of $\e$, namely $\l$, behaves almost like a local operator on sums of functions whose mass is `sufficiently' separated. It is exactly the regularity of $m$ that allows $\l$ to enjoy such a property. This result is encapsulated in the next lemma, which roughly states that the commutator operator $[L,\varphi(\cdot/r)]$ tends to zero as $r\to\infty$, for any Schwartz function $\varphi$. Here, the multiplication operator $f\mapsto\varphi f$ is defined for any distribution $f$ in the canonical sense. 

\begin{lemma}\label{zous}
For a Schwartz function $\varphi$, let $B_r\colon H^{\frac{s}{2}}\to H^{\frac{-s}{2}}$ be the commutator of the operators $L$ and $f\mapsto\varphi(\cdot/r) f$. Then
\begin{align*}
    \norm{B_r}{op}\to 0,\quad r\to\infty.
\end{align*}
\end{lemma}
\begin{proof}
Set $\varphi_r=\varphi(\cdot/r)$. Using the bound \eqref{moreconvenientbound}, we have for any $u,v\in H^{\frac{s}{2}}$,
\begin{align*}
|\langle [L,\varphi_r] u,v\rangle| &=\Big|\int_\R \int_\R \check{v}(\xi)\widehat{\vr}(t)\hat{u}(\xi-t)\big( m(\x) - m(\x-t)\big) dtd\xi \Big|\\
&\lesssim \int_\R |\widehat{\vr}(t)|\omega(t)\int_\R \cop{\xi}^{\frac{s}{2}}|\check{v}(\xi)|\cop{\xi-t}^{\frac{s}{2}}|\hat{u}(\xi-t)|d\xi dt\\
&\lesssim \underbrace{\int_\R|\hat{\varphi}(t)|\omega(t/r)dt}_{\let\scriptstyle\textstyle
    \substack{\gtrsim \norm{B_r}{op}}}\horm{u}\horm{v}.
\end{align*}
As $\omega$ is bounded above by a polynomial and $\lim_{t\to0}\omega(t)=0$, the statement of the lemma follows.
\end{proof}
We are now ready to prove that a dichotomized minimizing sequence can be `split' in two as described at the beginning of the section.
\begin{proposition}\label{spezial}
Suppose a minimizing sequence ${(u_k)\subset U_\m}$ dichotomizes, then there exist $0<\lambda<\mu$, and two sequences $(u_k^1)\subset U_\lambda$ and $(u_k^2)\subset U_{\mu-\lambda}$, so that
\begin{align*}
    \e(u_k^1) + \e(u_k^2)\to I_\m,\quad k\to\infty.
\end{align*}
\end{proposition}
\begin{proof} 
By the Concentration-Compactness principle, we can pick $(r_k)\subset\R^+$ with $r_k \to \infty$, and $(x_k)\subset\R$ so that
\begin{align}\label{splitting}
    \int_{X}|u_k(x-x_k)|^2dx\to\begin{cases} \lambda, &X=\{x: |x|\leq r_k\},\\
    0, &X=\{x:r_k\leq|x|\leq 2r_k\},\\
    \mu-\lambda,  &X=\{x:2r_k\leq|x|\},
    \end{cases}
\end{align}
as $k\to\infty$; without loss of generality, we assume  $x_k=0$ for all $k$. Next, we pick two smooth symmetrical functions $\varphi,\psi\colon \R\to[0,1]$, satisfying 
\(\varphi(x) = 1\) when $|x|\leq 1$, $\varphi = 0$ when \(|x|\geq 2\) and $\varphi^2+\psi^2=1$. We denote $\varphi_k$ and $\psi_k$ for $\varphi(\cdot/r_k)$ and $\psi(\cdot/r_k)$, and set \(v_k^1=\varphi_k u_k\) and \(v_k^2=\psi_k u_k\). By \eqref{splitting}, these function automatically satisfies 
\begin{align*}
    \q(v_k^1)\to\lambda,\quad
    \q(v_k^2)\to\mu-\lambda,\qquad k\to\infty.
\end{align*}
It is easily verified that if $\phi$ is Schwartz and symmetric, then $\cop{v,\phi u}=\cop{\phi v, u}$ for any $v\in H^{\frac{-s}{2}}$ and $u\in H^{\frac{s}{2}}$, and so we may write
\begin{align*}
    \l(\vun)- \langle Lu_k,\vn^2 u_k\rangle &= \langle [L,\vn]u_k,\vn u_k\rangle,\\
    \l(\vum)- \langle Lu_k,\pn^2 u_k\rangle &=  \langle [L,(1-\pn)]u_k,(1-\pn) u_k\rangle.
\end{align*}
By Lemma \ref{zous}, the RHS of these equations tend to zero, provided we can uniformly bound the $H^{\frac{s}{2}}$-norm of $u_k$, $\varphi_k u_k$ and $(1-\psi_k)u_k$ in $k$. By \eqref{hormbound}, this again is guaranteed if multiplication by $\varphi_k$ and $(1-\varphi_k)$ are uniformly bounded (in $k$) as operators on $H^{\frac{s}{2}}$. This is indeed true and follows by similar reasoning as in the proof of Prop. \ref{congestion}; it is trivially proven when $s/2\in\mathbb{N}_0$, and the result for general $s>0$ follows from interpolation. Thus $\l(\vun)+\l(\vum)-\l(u_k)\to0$, as $k\to\infty$.
Turning to $\n$, we have
\begin{align*}
    \n(\vun) + \n(\vum) - \n(u)&= \int_{r_k<|x|<2r_k}N(\vun)+N(\vum)-N(u_k)dx.
\end{align*}
By Prop. \ref{basicbounds}, we have $|N(x)|\lesssim x^2$, and so \eqref{splitting} guarantees the RHS of this equation to tend to zero as $k\to\infty$. As $(u_k)$ is a minimizing sequence, we conclude that
\begin{align*}\
    \e(\vun)+\e(\vum) \to I_\mu,
\end{align*}
for $k\to\infty$. By the same reasoning as before, the $H^{\frac{s}{2}}$-norm of $v^1_k$ and $v^2_k$ is uniformly bounded in $k$, and so by Corollary \ref{continuityofe} the proposition is proved for the two sequences $\un=\vun\sqrt{\lambda/\q(\vun)}$ and $\um=\vum\sqrt{(\mu-\lambda)/\q(\vum)}$.
\end{proof}

With these two results at hand, we can exclude dichotomy; picking $\mu_*>0$ so that $\mu\mapsto I_\mu$ is strictly subadditive and assuming $(u_k),(u_k^1)$ and $(u_k^2)$ to be as in the previous proposition, we arrive at the contradiction
\begin{align*}
    I_\mu=\lim_{k\to\infty}\e(u_k^1) + \e(u_k^2)\geq\liminf_{k\to\infty}\e(u_k^1) + \liminf_{k\to\infty}\e(u_k^2)\geq I_{\lambda} + I_{\mu-\lambda}.
\end{align*}
\begin{corollary}
Provided $\mu_*>0$ is sufficiently small, no minimizing sequence of \eqref{problem} has a subsequence for which dichotomy occurs in accordance with Theorem \ref{concentration}.
\end{corollary}

\section{Solutions from concentrated minimizing sequences}\label{convconssection}
Theorem \ref{concentration} provided us with the three possible phenomena that could occur for a minimizing sequence of \eqref{problem}; the previous two sections excluded vanishing and dichotomy, and so it remains to see that we can construct a minimizer from a \textit{concentrating} minimizing sequence. This is straight forward:

\begin{proposition}
Provided $\mu_*>0$ is sufficiently small, any minimizing sequence $(u_k)\subset U_\mu$ of \eqref{problem} admits a subsequence converging in $L^2$-norm to a minimizer $u\in U_\mu$.
\end{proposition}
\begin{proof}
For $\mu_*$ sufficiently small, the two preceding sections guarantees that $(u_k)$ admits a subsequence, again denoted $(u_k)$, that concentrates in accordance with Theorem  \ref{concentration}. Without loss of generality, we assume $(u_k)$ to consist solely of near minimizers and shifted appropriately to concentrate about zero ($x_k=0$ for all $k$). By the Kolmogorov-Riesz-Fréchet compactness theorem, $(u_k)$ is relatively compact in $L^2$, as it is bounded, concentrated about zero and uniformly continuous with respect to translation:
\begin{align*}
    \norm{u_k(\cdot+y)-u_k(\cdot)}{2}&=\norm{(e^{-i(\cdot)y}-1)\hat{u}_k}{2}\\
    &\leq \norm{(e^{-i(\cdot)y}-1)\cop{\cdot}^{\frac{-s}{2}}}{\infty}\horm{u_k}\\
    &\to 0,
\end{align*}
uniformly in $k$ as $y\to0$, as guaranteed by \eqref{hormbound}. We conclude that $(u_k)$ admits a subsequence, yet again denoted $(u_k)$, so that $u_k\to u$, for some $u\in L^2$ with $\q(u)=\mu$. We now demonstrate that $u$ is a minimizer of \eqref{problem}. As the positive functions $m(\cdot)|\hat{u}_k|^2$ converges locally in measure to $m(\cdot)|\hat{u}|^2$, Fatou's lemma implies
\begin{align*}
    \l(u)\leq\liminf_{k\to\infty}\l(u_k).
\end{align*}
Using the Fréchet derivative (Prop. \ref{frechetderivative}) of $\n$, and that $|n(x)|\lesssim|x|$, we also obtain
\begin{align*}
    |\n(u)-\n(u_k)|&=\Big|\int_0^1\int_\R n(tu+(1-t)u_k)(u-u_k) dxdt\Big|\\
    &\lesssim\int_0^1\norm{tu+(1-t)u_k}{2}\norm{u-u_k}{2}dt\\
    &\to0,
\end{align*}
as $k\to\infty$. We now have $I_\mu\leq\e(u)\leq\liminf_{k\to\infty}\e(u_k)= I_\mu$.
\end{proof}

Not only is a  minimizer of \eqref{problem} a solutions of \eqref{original}, we are also provided some additional control over the respective velocity $\nu$, as described in the next proposition. 

\begin{proposition}\label{frommintosol}
Any minimizer $u\in U_\mu$ of the minimization problem \eqref{problem},  solves \eqref{original} in distribution sense, with velocity
  $\nu = \cop{\e'(u),u}/2\mu$.
Provided $\mu_*>0$ is small enough, we additionally have $-\nu\simeq\mu^\beta$.
\end{proposition}
\begin{proof}
As the feasible set $U_\mu$ is a Hilbert submanifold of $H^{\frac{s}{2}}$, it follows that there must be a Lagrange multiplier $\nu\in\R$ (depending on the minimizer $u$), so that
\begin{align}\label{solvesTheEquationInWeakSense}
    \e'(u)-\nu\q'(u)=0,
\end{align}
in $H^{-\frac{s}{2}}$. In particular, if we pair \eqref{solvesTheEquationInWeakSense} with $u$ and insert for $\q'$ we obtain
\begin{align*}
   \nu= \frac{\cop{\e'(u),u}}{2\mu},
\end{align*}
and so we attain the first part of the proposition. For the latter, note that 
\begin{align*}
    n(u)u=(2+p)N(u) + n_r(u)u-(2+p)N_r(u),
\end{align*}
and as argued in the proof of Prop. \ref{subadditive}, we have
\begin{align*}
    \int_\R n_r(u)u-(2+p)N_r(u)dx=o(\mu^{1+\beta}).
\end{align*}
Then
\begin{align*}
    \cop{\e'(u),u}&= \cop{Lu,u}-\cop{n(u),u}\\
    &=2\l(u)-(2+p)\n(u) + o(\mu^{1+\b})\\
    &=2I_\mu -p\n(u)+ o(\mu^{1+\b})\\
    &<-C\mu^{1+\beta}+o(\mu^{1+\b}),
\end{align*}
for some fixed $C>0$, by Prop. \ref{bounds} and \eqref{equalbounds}. Thus, for a sufficiently small $\mu_*>0$ we obtain $-\nu\gtrsim\mu^\beta$ when $\mu\in(0,\mu_*)$. The upper bound on $-\nu$ follows trivially from
\begin{align*}
    -\nu\lesssim\frac{1}{\mu}\Big(\l(u)+\norm{u}{2+p}^{2+p}\Big)\lesssim\mu^\beta,
\end{align*}
where we used $|n(x)x|\lesssim|x|^{2+p}$ and \eqref{equalbounds}.
\end{proof}

\section{Regularity of solutions}\label{regularitysection}
Before moving on, we summarize what has been proved so far. For the class of equations \eqref{original} that satisfies the assumptions (A) and (B) (see subsection \ref{assumptions}) and the `auxiliary' assumptions (C\textsubscript{1}) and (C\textsubscript{2}) (see subsection \ref{the method}), we have proved all parts of Theorem \ref{main}, except the estimate $\norm{u}{H^{1+s}}^2\lesssim\mu$. By Lemma \ref{equivalenceOfAssumptions}, when this estimate is proven, the theorem automatically holds in the case when only (A) and (B) are satisfied. Hence, we now introduce the final piece, concluding the proof of Theorem \ref{main}. 

\begin{proposition}\label{regularityofsolutions}
Provided $\mu_*>0$ is sufficiently small, minimizers $u\in U_\mu$ of \eqref{problem} satisfies
\begin{align*}
\norm{u}{H^{1+s}}^2 \lesssim\mu.
\end{align*}
\end{proposition}
\begin{proof}
By Prop. \ref{frommintosol}, minimizers are solutions of \eqref{original}, and so by a little rewriting, we have  
\begin{align}\label{regularitystart}
     \underbrace{(L-\nu+1)u}_{\let\scriptstyle\textstyle\substack{\Lambda_{\nu} u}} =\underbrace{n(u)+u}_{\let\scriptstyle\textstyle\substack{\eta(u)}}.
\end{align}
Proposition \ref{frommintosol} also guarantees that ${-\nu+1>\delta}$ for a positive constant $\delta$ independent of $\mu\in(0,\mu_*)$, provided $\mu_*>0$ is small enough. The inverse of $\Lambda_{\nu}$ then defines a bounded linear Fourier multiplier, $\Lambda_{\nu}^{-1}\colon  H^{\alpha}\to H^{\alpha+s}$ for any $\alpha\in\R$, whose norm has the upper bound
\begin{align*}
    \norm{\Lambda_{\nu}^{-1}}{H^{\alpha}\to H^{\alpha+s}}=\sup_{\xi\in\R}\frac{\cop{\xi}^s}{m(\xi)-\nu+1}\leq\sup_{\xi\in\R}\frac{\cop{\xi}^s}{m(\xi)+\delta}\eqqcolon C.
\end{align*} Clearly $C$ is independent of $\mu\in(0,\mu_*)$. We also note that $T_\eta\colon u\mapsto\eta(u)$, is a bounded operator on $H^{\alpha}$, whenever $0\leq\alpha\leq1$, as $\eta$ is globally Lipschitz continuous with $\eta(0)=0$. Looking back at \eqref{regularitystart}, a minimizer $u\in U_{\mu}$ satisfies
\begin{align}\label{bootstrap}
    \norm{u}{H^{\alpha+s}}=\norm{\Lambda_{\nu}^{-1}\circ T_{\eta}(u)}{H^{\alpha+s}}\lesssim\norm{u}{H^{\alpha}},
\end{align}
whenever $0\leq\alpha\leq1$ (where the implicit constant in \eqref{bootstrap} can depend on $\alpha$). We now obtain the desired conclusion by the following `bootstrap' argument. Pick $k\in \N$ and  $0\leq r<s$ so that $1 +s = ks + r$. By a (finite) repeated use of \eqref{bootstrap}, we obtain
\begin{align*}
    \norm{u}{H^{1+s}}=\norm{u}{H^{ks+r}}\lesssim\norm{u}{H^{(k-1)s+r}}\lesssim\cdots\lesssim \norm{u}{H^{r}}\leq \norm{u}{H^s}\lesssim\norm{u}{L^2},
\end{align*}
and so we are done.
\end{proof}
\subsection{Further regularity}
We conclude this paper with a regularity result on the solutions we have constructed. Clearly, if equation \eqref{bootstrap} was satisfied for large $\alpha$, we could (as done in the previous proof) bootstrap to corresponding regularity. It is ultimately the regularity of $n$ that determines how large $\alpha$ can be in \eqref{bootstrap}. In \cite{BourdaudSickel}, the authors prove that for any $\gamma>3/2$, the composition operator $T_f:u\mapsto f(u)$ maps $H^{\gamma}$ to itself if, and only if, $f(0)=0$ and $f\in H^{\gamma}_{loc}$; in particular, if we restrict $\norm{u}{\infty}<R<\infty$, then we have 
\begin{align}\label{compositionoperatorbound}
    \norm{f(u)}{H^{\alpha}}\leq C\norm{u}{H^{\alpha}},
\end{align}
for some constant $C$ depending only on $f,R$ and $\alpha\in(\frac{3}{2},\gamma]$. Moreover, using the result of \cite{Moussai}, we can extend the inequality \eqref{compositionoperatorbound} to the case $\alpha\in[1,\gamma]$ (still with $\gamma>3/2$). It is now an easy task to improve the regularity of our solutions when $n\in H^{\alpha_*}_{loc}$ for some $\alpha_*>3/2$; note that functions in these spaces are necessarily locally Lipschitz continuous. We present the final proposition of this paper. 
\begin{proposition}\label{furtherRegularity}
If $n\in H^{\alpha_*}_{loc}$ with $\alpha_*>3/2$, then the solutions $u$ of \eqref{original} provided by Theorem \ref{main}, satisfies
\begin{align*}
     \norm{u}{H^{\alpha_*+s}}\lesssim\norm{u}{2}.
\end{align*}
\end{proposition}
\begin{proof}
Looking back at \eqref{bootstrap}, this equation is now valid for $0\leq\alpha\leq\alpha_*$. This follows from the previous discussion as: 1) $\eta\in H^{\alpha_*}_{loc}$ with $\eta(0)=0$, and 2) by Theorem \ref{main} we have a uniform upper bound on the $L^{\infty}$-norm of our solutions $u$ ($\mu_*$ is fixed). The result is then attained by a similar bootstrap argument as the one used in the proof of Prop. \ref{regularityofsolutions}. 
\end{proof}
\section{Acknowledgements}
The author would like to thank the referee for constructive feedback and Vincent Duchêne for helpful comments on an earlier version of this manuscript. 
\bibliographystyle{siam}
\bibliography{references}
\end{document}